\documentclass{amsart}
\usepackage{multirow}
\usepackage{fullpage}
\usepackage{graphicx}
\usepackage{longtable}
\usepackage{amssymb}
\usepackage{hyperref}

\newcommand{\ZZ}{\mathbb{Z}}
\newcommand{\QQ}{\mathbb{Q}}

\newcommand{\FF}{\mathbb{F}}
\newcommand{\PP}{\mathbb{P}}
\newcommand{\OO}{\mathcal{O}}
\newcommand{\p}{\mathfrak{p}}
\newcommand{\lra}{\longrightarrow}
\newcommand{\Frob}{\mathop{\rm Frob}}
\DeclareMathOperator{\tr}{tr}
\long\def\ppp#1||{\parbox{13cm}{\rule[3.7mm]{0mm}{1pt}$#1$\rule[-1.8mm]{0mm}{1pt}}}
\theoremstyle{plain}
\newtheorem{thm}{Theorem}

\newtheorem{prop}[thm]{Proposition}
\newtheorem{remark}[thm]{Remark}

\numberwithin{thm}{section}
\numberwithin{equation}{section}

\begin{document}

\title{Hilbert modularity of some double octic Calabi--Yau threefolds}



\author{S\l awomir Cynk}
\address{Institute of Mathematics, Jagiellonian University, ul. \L
ojasiewicza 6,30-348 Krak\'ow,
Poland
}
\email{slawomir.cynk@im.uj.edu.pl}

\author{Matthias Sch\"utt}
\address{Institut f\"ur Algebraische Geometrie, 
Leibniz Universit\"at  Hannover, Welfengarten 1, 30167 Hannover, Germany
\newline\indent
Riemann Center for Geometry and Physics, 
  Leibniz Universit\"at Hannover, 
  Appelstrasse 2, 30167 Hannover, Germany}
\email{schuett@math.uni-hannover.de}


\author{Duco van Straten}
\address{Institut f\"ur Mathematik,
FB 08 -- Physik, Mathematik und Informatik,
Johannes Gutenberg-Universit\"at,
Staudingerweg 9,
55128 Mainz}
\email{straten@mathematik.uni-mainz.de}

\thanks{The first named author was partially supported by the National Science
  Center grant no. 2014/13/B/ST1/00133. This
  research was supported in part by PLGrid Infrastructure.
  Partial funding by the grant 346300 for IMPAN from
  the Simons Foundation and the matching 2015-2019 Polish MNiSW fund
  is gratefully acknowledged.
}

\date{October 10, 2018}

\begin{abstract}
We exhibit three double octic Calabi--Yau threefolds
over the quadratic fields $\QQ(\sqrt 2), \QQ(\sqrt 5), \QQ(\sqrt{-3})$
and prove their modularity.
The non-rigid threefold $\QQ(\sqrt 2)$ has two conjugate Hilbert modular forms 
of weight $[4,2]$ and $[2,4]$ attached
while the two rigid threefolds correspond to a Hilbert modular form of weight $[4,4]$
and to the twist of the restriction of a classical modular form of weight $4$.
\end{abstract}

\maketitle

\section{Introduction}
\label{s:intro}

After the modularity of elliptic curves over $\QQ$ was proven,
much emphasis has been put on Calabi--Yau threefolds.
Indeed, for rigid Calabi--Yau threefolds over $\QQ$,
the modularity has been established in the meantime 
independently in \cite{D}, \cite{GY}.
Over number fields other than $\QQ$, however,
there seems to be only one (Hilbert) modular example to date (outside the CM case),
the Consani--Scholten quintic from \cite{CoS}, \cite{DPS}.

This paper will provide three more examples of modular Calabi-Yau threefolds,
each of which is defined over some quadratic field $K$.
In detail, we will exhibit
\begin{enumerate}
\item
a non-rigid Calabi--Yau threefold $X$ over $\QQ$ (with $b_3(X)=4$)
admitting an endomorphism over $\QQ(\sqrt 2)$
which can be used to split $H^3(X)$ into two 2-dimensional eigenspaces;
the corresponding Galois representations are proved to correspond to
a Hilbert modular form over $\QQ(\sqrt 2)$ (Theorem \ref{thm1}, Remark \ref{rem'});
\item
a rigid Calabi--Yau threefold $Y$ over $\QQ(\sqrt 5)$
such that the Galois representation on $H^3(Y)$
corresponds to a Hilbert modular form over $\QQ(\sqrt 5)$ of weight $[4,4]$ (Theorem \ref{thm2});
\item
a rigid Calabi--Yau threefold $Z$ over $\QQ(\sqrt{-3})$
such that the Galois representation on $H^3(Z)$
corresponds to a twist of the restriction of a classical modular form of weight $4$ and level $72$
(Theorem \ref{thm3}).
\end{enumerate}
The Calabi--Yau threefolds will be constructed as crepant resolutions
of certain double octics.
More precisely, we will choose the branch locus to consist of 8 distinct planes.
The construction, following \cite{Mey}, will be reviewed in Section \ref{s:octics}.
Then the proof of the mentioned results amounts to comparing
the two-dimensional Galois representations attached
to the motive of the third cohomology of the Calabi--Yau threefolds (or a given submotive in the first case)
on the one hand
and to the Hilbert or classical modular forms in question on the other.
In practice, this can be achieved by working with the $2$-adic representations
and applying a method going back to Faltings and Serre and worked out in detail by Livn\'e in \cite{L}.
This technique will be explained, with a view towards the given base fields, 
in Section \ref{s:Galois}.
In essence, it reduces the proof of modularity
to comparing a few traces and determinants of the Galois representations at certain Frobenius elements;
these, in turn, can be obtained from extensive point counting
using the Lefschetz fixed point formula,
and from Hilbert modular forms calculations as incorporated in MAGMA.

\subsection*{Acknowledgements}

We are indebted to John Voight for very helpful explanations.

\section{Faltings--Serre--Livn\'e method}

\label{s:Galois}

In this section we study special cases of \cite[Thm. 4.3.]{L};
they will be instrumental in establishing  the modularity of the three two-dimensional
Galois representations attached to the Calabi--Yau threefolds $X, Y, Z$
to be introduced in the next sections.

Throughout this section, 
the set-up comprises continuous two-dimensional $2$-adic Galois representations
of the absolute Galois group of a specified  number field $K$
which are
unramified outside a given finite set $S$ of prime ideals in the integer ring
$\mathcal O_{K}$ of $K$.  
For simplicity, we list prime ideals just by a single generator.
The norms will be included in the tables to follow.

\begin{prop}\label{prop1.1}
Let $K=\QQ[\sqrt2]$ and $E=\QQ_{2}[\sqrt2]$
and let $\mathcal P:=\sqrt2\,\ZZ_{2}$ be the maximal
ideal of the ring of integers of $E$. Let $S:=\{\sqrt{2}, 3\}$ and
\begin{eqnarray*}
T&=&\{5, 11, \sqrt2+3,\
\sqrt2-3,\
3\sqrt2-1,\
\sqrt2+5,\
\sqrt2-5,\
4\sqrt2-1,\
4\sqrt2+1,\
5\sqrt{2}-3,\
\sqrt2-7,\
\sqrt2+7,\\
&&4\sqrt2-11,\
1-7\sqrt{2}\}\\
  U&=&\{5, 11, 13,\ 
\sqrt2-3,\
 3\sqrt2-1,\
\sqrt2-5,\
4\sqrt2-1,\
5\sqrt2-3
\}
\end{eqnarray*}
be two sets of primes in $\mathcal O_{K}$.
Suppose that $\rho_{1},\rho_{2}:\operatorname{Gal}(\bar
K/K)\lra\operatorname{GL}_{2}(E)$ are continuous Galois
representations unramified outside $S$ 
and satisfying 
\begin{enumerate}
\item [1.] 
$\operatorname{Tr}(\rho_{1}(Frob_{\mathfrak p})) \equiv
  \operatorname{Tr}(\rho_{2}(Frob_{\mathfrak p})) \equiv 0
  \pmod{\mathcal P}$  for
  $\mathfrak p\in U$,
\item [2.]
  $\det(\rho_{1}) \equiv
  \det(\rho_{2})  \pmod{\mathcal P}$,
\item [3.] $\operatorname{Tr}(\rho_{1}(Frob_{\mathfrak p})) =
  \operatorname{Tr}(\rho_{2}(Frob_{\mathfrak p})) $
and
  $\det(\rho_{1}(Frob_{\mathfrak p})) =
  \det(\rho_{2}(Frob_{\mathfrak p})) $
  for $\mathfrak p\in T$.
\end{enumerate}
Then $\rho_{1}$ and $\rho_{2}$ have isomorphic semisimplifications.
\end{prop}

\begin{proof}
  Following the arguments of \cite{L} we  first verify that
  assumption  1. implies that $\operatorname{Tr}(\rho_{1}) \equiv
  \operatorname{Tr}(\rho_{2}) \equiv 0
  \pmod{\mathcal P}$. 
  Indeed, suppose that  $\operatorname{Tr}(\rho_{i})
  \not\equiv 0\pmod{\mathcal P}$ and denote by  $L/K$
  the Galois extension cut out by the kernel
  $\operatorname{Ker}{\bar{\rho}_{i}}$ of the reduction
  $\bar{\rho}_{i}$ of $\rho_{i}$ modulo $\mathcal P$.  
  By inspection, we have im$(\bar\rho_i)\subseteq$ GL$_2(\FF_2)$
  where the elements of odd trace are exactly those of order $3$ 
  (which by assumption will correspond to some Frobenius elements).
  Hence, the Galois group of the extension $L/K$ is isomorphic to $S_{3}$ or
   $C_{3}$, so it is the Galois closure of a degree 3 extension
   $M/K$. 

   Then $M$
   is a degree 6 extension of $\mathbb Q$ unramified outside
   $\{2,3\}$. The database \cite{JR} lists 398 such fields presented
   as a splitting field of a monic degree 6 polynomial with rational
   coefficients. The assumption  that $M$ contains the subfield
   $K=\QQ[\sqrt2]$ implies that the minimal polynomial of any primitive
   element of the extension $M/\QQ$ factors over $\QQ[\sqrt2]$. 
   We check that exactly 25 
    of the  398 polynomials from \cite{JR}  satisfy this condition. 
   For each of these 25 degree 6 polynomials, we determine a prime integer $p$
   such that the reduction of the degree 3 polynomial over $K$ modulo a prime $\mathfrak p $ in $\OO_K$ over
   $p$  stays irreducible over
   $\OO_K/\mathfrak p \cong \FF_{\mathfrak p}$. 
   We list these data below. 
   \bgroup\small\parindent=10mm

  \(\displaystyle {x}^{6}-2\,{x}^{3}-1 = \left( {x}^{3}+\sqrt {2}-1 \right) \times \left( {x}^{3}-\sqrt {2}-1 \right), \ p=5  \)

\(\displaystyle {x}^{6}-12\,{x}^{4}+36\,{x}^{2}-8 = \left( {x}^{3}-6\,x-2\,\sqrt {2} \right) \times \left( {x}^{3}-6\,x+2\,\sqrt {2} \right), \ p=5  \)

\(\displaystyle {x}^{6}-2 = \left( {x}^{3}+\sqrt {2} \right) \times \left( {x}^{3}-\sqrt {2} \right), \ p=7  \)

\(\displaystyle {x}^{6}-4\,{x}^{3}+2 = \left( {x}^{3}+\sqrt {2}-2 \right) \times \left( {x}^{3}-\sqrt {2}-2 \right), \ p=5  \)

\(\displaystyle {x}^{6}+6\,{x}^{4}+9\,{x}^{2}-8 = \left( {x}^{3}+3\,x+2\,\sqrt {2} \right) \times \left( {x}^{3}+3\,x-2\,\sqrt {2} \right), \ p=11  \)

\(\displaystyle {x}^{6}+6\,{x}^{4}+9\,{x}^{2}-2 = \left( {x}^{3}+3\,x-\sqrt {2} \right) \times \left( {x}^{3}+3\,x+\sqrt {2} \right), \ p=5  \)

\(\displaystyle {x}^{6}-6\,{x}^{4}-6\,{x}^{3}+12\,{x}^{2}-36\,x+1 = \left(
   {x}^{3}+3\,\sqrt {2}{x}^{2}+6\,x+2\,\sqrt {2}-3 \right)\)

\(\displaystyle   \rule{4.02cm}{0cm} \times\left( {x}^{3}-3\,\sqrt {2}{x}^{2}+6\,x-2\,\sqrt {2}-3 \right), \ p=7  \)

\(\displaystyle {x}^{6}-18 = \left( {x}^{3}-3\,\sqrt {2} \right) \times \left( {x}^{3}+3\,\sqrt {2} \right), \ p=7  \)

\(\displaystyle {x}^{6}-6\,{x}^{4}-12\,{x}^{3}+12\,{x}^{2}-72\,x+28 = \left(
   {x}^{3}-3\,\sqrt {2}{x}^{2}+6\,x-2\,\sqrt {2}-6 \right) 
   \)

   \(\displaystyle \rule{4.02cm}{0cm}\times\left( {x}^{3}+3\,\sqrt {2}{x}^{2}+6\,x+2\,\sqrt {2}-6 \right), \ p=13  \)

\(\displaystyle {x}^{6}-6\,{x}^{3}-9 = \left( {x}^{3}-3\,\sqrt {2}-3 \right) \times \left( {x}^{3}+3\,\sqrt {2}-3 \right), \ p=5  \)

\(\displaystyle {x}^{6}-6\,{x}^{4}-4\,{x}^{3}+9\,{x}^{2}+12\,x-14 =
\left( {x}^{3}-3\,x-3\,\sqrt {2}-2 \right) \times \left( {x}^{3}-3\,x+3\,\sqrt {2}-2 \right), \ p=5  \)

\(\displaystyle {x}^{6}-18\,{x}^{4}-12\,{x}^{3}+81\,{x}^{2}+108\,x+18 = \left( {x}^{3}-9\,x+3\,\sqrt {2}-6 \right) \times \left( {x}^{3}-9\,x-3\,\sqrt {2}-6 \right), \ p=5  \)

\(\displaystyle {x}^{6}+6\,{x}^{4}-4\,{x}^{3}-9\,{x}^{2}+12\,x-4 = \left(
   {x}^{3}+3\,\sqrt {2}x+3\,x-2\,\sqrt {2}-2 \right)\)

\(\displaystyle \rule{4.02cm}{0cm}  \times \left( {x}^{3}-3\,\sqrt {2}x+3\,x+2\,\sqrt {2}-2 \right), \ p=31  \)

\(\displaystyle {x}^{6}+6\,{x}^{4}-4\,{x}^{3}+9\,{x}^{2}-12\,x-4 = \left(
   {x}^{3}+3\,x+2\,\sqrt {2}-2 \right) \times \left( {x}^{3}+3\,x-2\,\sqrt {2}-2 \right), \ p=23  \)

\(\displaystyle {x}^{6}-6\,{x}^{4}-4\,{x}^{3}+9\,{x}^{2}+12\,x-4 = \left( {x}^{3}-3\,x-2\,\sqrt {2}-2 \right) \times \left( {x}^{3}-3\,x+2\,\sqrt {2}-2 \right), \ p=5  \)

\(\displaystyle {x}^{6}-12\,{x}^{3}+18 = \left( {x}^{3}+3\,\sqrt {2}-6 \right) \times \left( {x}^{3}-3\,\sqrt {2}-6 \right), \ p=5  \)

\(\displaystyle {x}^{6}-12\,{x}^{3}-36 = \left( {x}^{3}+6\,\sqrt {2}-6 \right) \times \left( {x}^{3}-6\,\sqrt {2}-6 \right), \ p=5  \)

\(\displaystyle {x}^{6}-6\,{x}^{4}-4\,{x}^{3}-9\,{x}^{2}-12\,x-4 = \left(
   {x}^{3}-3\,\sqrt {2}x-3\,x-2\,\sqrt {2}-2 \right) \times\)

  \(\displaystyle \rule{4.02cm}{0cm}\left( {x}^{3}+3\,\sqrt {2}x-3\,x+2\,\sqrt {2}-2 \right), \ p=41  \)

\(\displaystyle {x}^{6}-8\,{x}^{3}-18\,{x}^{2}-48\,x-16 = \left( {x}^{3}-3\,\sqrt
   {2}x-4\,\sqrt {2}-4 \right)\times \left( {x}^{3}+3\,\sqrt {2}x+4\,\sqrt {2}-4 \right), \ p=7  \)

\(\displaystyle {x}^{6}+6\,{x}^{4}-12\,{x}^{3}+9\,{x}^{2}-36\,x+28 = \left( {x}^{3}+3\,x-2\,\sqrt {2}-6 \right) \times \left( {x}^{3}+3\,x+2\,\sqrt {2}-6 \right), \ p=17  \)

\(\displaystyle {x}^{6}-8\,{x}^{3}-18\,{x}^{2}+24\,x+8 = \left( {x}^{3}-3\,\sqrt {2}x+2\,\sqrt {2}-4 \right) \times \left( {x}^{3}+3\,\sqrt {2}x-2\,\sqrt {2}-4 \right), \ p=13  \)

\(\displaystyle {x}^{6}-16\,{x}^{3}-18\,{x}^{2}+48\,x+32 = \left( {x}^{3}+3\,\sqrt {2}x-4\,\sqrt {2}-8 \right) \times \left( {x}^{3}-3\,\sqrt {2}x+4\,\sqrt {2}-8 \right), \ p=11  \)

\(\displaystyle {x}^{6}-18\,{x}^{4}-36\,{x}^{3}-81\,{x}^{2}-108\,x+36 = \left(
   {x}^{3}-9\,\sqrt {2}x-9\,x-12\,\sqrt {2}-18 \right) \)

  \(\displaystyle \rule{4.02cm}{0cm}\times \left( {x}^{3}+9\,\sqrt {2}x-9\,x+12\,\sqrt {2}-18 \right), \ p=11 \)

\(\displaystyle {x}^{6}-18\,{x}^{4}-12\,{x}^{3}+81\,{x}^{2}+108\,x-36 = \left( {x}^{3}-9\,x-6\,\sqrt {2}-6 \right) \times \left( {x}^{3}-9\,x+6\,\sqrt {2}-6 \right), \ p=11  \)

\(\displaystyle {x}^{6}-18\,{x}^{4}-36\,{x}^{3}+81\,{x}^{2}+324\,x+252 = \left( {x}^{3}-9\,x-6\,\sqrt {2}-18 \right) \times \left( {x}^{3}-9\,x+6\,\sqrt {2}-18 \right), \ p=11\)  

\egroup
Given $M$ and $\p$ as above, it follows 
that any element in the conjugacy class of   $\operatorname{Frob}_{\mathfrak p}$ 
in   $\operatorname{Gal}(L/K)$ 
has order 3;
consequently
   Tr$(\rho_i(\operatorname{Frob}_{\mathfrak p}))\equiv 1\pmod{\mathcal P}$, 
   contradicting our assumptions.
Thus we see that the set $U$ was indeed chosen in such a way that
condition 1. implies that
$$
\operatorname{Tr}(\rho_{1})\equiv \operatorname{Tr}(\rho_{2})\equiv
0\pmod{\mathcal P}.
$$
The  traces being even is the key ingredient to apply  \cite[Thm.~4.3]{L}.
To this end, let $K_{S}$ be the compositum of all quadratic extensions of $K$
unramified outside $S$.
Since the ring $\mathcal O_{K}$ is a unique factorization domain, 
the compositum $K_S$ is obtained by extracting square roots
of generators of $\mathcal
O_{K}^{*}$ and 
of prime elements $\alpha\in\mathcal O_{K}$ with norm $N_{K}(\alpha)$
divisible only by elements of $S$.  
Presently, generators of $K_S/K$ can be taken as
$$\sqrt{-1}, \sqrt[4]2, \sqrt{\sqrt 2-1}, \sqrt 3.
$$
One computes the table of quadratic characters
$\operatorname{Gal}(K_{S}/K)\lra (\ZZ/2)^{4}$ at the primes from $T$ as follows:
\[
  \begin{array}{r|r|c|c|c|c||r|r|c|c|c|c}
    \mathfrak p&N(\mathfrak p)&\sqrt{2}&\ 3\ &\ -1\
    &\sqrt{2}-1&\mathfrak p&N(\mathfrak p)&\sqrt{2}&\ 
3\ &\ -1\ &\sqrt{2}-1\\\hline
5&25&1&0&0&0&4\sqrt2-1&31&0&1&1&0\\
11&121&0&0&0&1
&4\sqrt2+1&31&1&1&1&1\\
\sqrt2+3&7&0&1&1&1

&5\sqrt2-3&41&1&1&0&0\\
\sqrt2-3&7&1&1&1&0
&\sqrt2-7&47&0&0&1&0\\
3\sqrt2-1&17&1&1&0&1&\sqrt2+7&47&1&0&1&1\\
\sqrt2+5&23&0&0&1&1&4\sqrt2-11&89&0&1&0&1\\
\sqrt2-5&23&1&0&1&0&1-7\sqrt2&97&1&0&0&1\\
  \end{array}
\]
From the  table we infer that 
the image of the Frobenius elements $\operatorname{Frob}_{t}, t\in T$,
contains 14 different non-zero elements, hence it is  non--cubic
in the terminology of \cite[Def.\ 4.1]{L}
(see e.g.\ \cite[p.~53]{S-Dipl}). 
Thus the assertion that
 the Galois representations $\rho_1, \rho_2$ 
 have isomorphic semisimplifications
follows from \cite[Thm.~4.3]{L}. 
\end{proof}

Note in particular that Proposition \ref{prop1.1}
implies that $\rho_1, \rho_2$ 
have the same $L$-series,
a feature which will be centrally used in the proof of the modularity results stated in the introduction.
The next two propositions concern the same kind of problem
for different base fields and adjusted ramification sets.
Since the arguments are very similar,
we give only the essential ingredients.

\begin{prop}
\label{prop2}
Let $K=\QQ[\sqrt5]$, $E=\QQ_{2}$
and let $\mathcal P:=2\ZZ_{2}$ be the maximal
ideal of the ring of integers of $E$. Let $S:=\{2\}$
and
\begin{eqnarray*}
T&=&\{3, 13, 
\sqrt5+4, 2\sqrt5+7, 
     \sqrt5+6, \sqrt5-6, 2\sqrt5+9\}
\end{eqnarray*}
be two sets of primes in $\mathcal O_{K}$.
Suppose that $\rho_{1},\rho_{2}:\operatorname{Gal}(\bar
K/K)\lra\operatorname{GL}_{2}(E)$ are continuous Galois
representations unramified outside $S$ and
satisfying 
\begin{enumerate}
\item [1.] $\operatorname{Tr}(\rho_{1}(Frob_{3})) \equiv
  \operatorname{Tr}(\rho_{2}(Frob_{3})) \equiv 0
  \pmod{\mathcal P}$,
\item [2.]
  $\det(\rho_{1}) \equiv
  \det(\rho_{2})  \pmod{\mathcal P}$,
\item [3.] $\operatorname{Tr}(\rho_{1}(Frob_{\mathfrak p})) =
  \operatorname{Tr}(\rho_{2}(Frob_{\mathfrak p})) $
and
  $\det(\rho_{1}(Frob_{\mathfrak p})) =
  \det(\rho_{2}(Frob_{\mathfrak p}))  $ for
  $\mathfrak p\in T$.
\end{enumerate}
Then $\rho_{1}$ and $\rho_{2}$ have isomorphic semisimplifications.
\end{prop}

\begin{proof}
By \cite{JR} there are 106 degree 6 extensions of $\QQ$ unramified
outside $\{2,5\}$. Only one of them contains $\QQ[\sqrt5]$; it is the
splitting field of the following polynomial:
  \[{x}^{6}-2\,{x}^{5}-2\,x-1=
 \left({x}^{3}-x^{2}+\tfrac12(-1+\sqrt {5})x+\tfrac12(-1+\sqrt {5})
 \right)
 \left( {x}^{3}-x^{2}-\tfrac12(1+\sqrt {5})x-\tfrac12(1+\sqrt {5}) \right)\]
As each  cubic polynomial over $\ZZ[\sqrt{5}]$
is irreducible modulo $3$,
 assumption 1. implies that
$\operatorname{Tr}(\rho_{1}) \equiv \operatorname{Tr}(\rho_{2}) \equiv
0 \pmod{\mathcal P}$ as required.

The compositum $K_{S}$ of quadratic extensions of $K$ unramified outside $S$
is obtained from the three quadratic extensions $K[\sqrt2],
K[\sqrt{-1}], K[\sqrt{\tfrac12(\sqrt5-1)}]$.
The table of characters is computed as follows:
  \[\begin{array}{r|r|c|c|c||r|r|c|c|c}
  \mathfrak p&N(\mathfrak p)&\ \ 2\ \ &\ -1\ &\tfrac12(\sqrt5-1)&
  \mathfrak p&N(\mathfrak p)&\ \ 2\ \ &\ -1\
    &\tfrac12(\sqrt5-1)\\\hline
   3&9&1&1&0&
\sqrt5+6&31&1&0&0   
   \\
13&169&1&1&1&
\sqrt5-6&31&1&0&1
\\
\sqrt5+4&11&0&0&1
&
2\sqrt5+9&61&0&1&0
\\
2\sqrt5+7&29&0&1&1
&&&&&
  \end{array}
\]
From the table, we infer that the image 
of the Frobenius elements $\operatorname{Frob}_{t}, t\in T$
 equals $(\ZZ/2)^{3}\setminus\{0\}$; in particular, it
is non-cubic, and the Proposition follows from \cite[Thm. 4.3]{L}
as before.
\end{proof}

\begin{prop}\label{prop1.3}
Let $K=\QQ[\sqrt{-3}]$, $E=\QQ_{2}$
and let $\mathcal P:=2\ZZ_{2}$ be the maximal
ideal of the ring of integers of $E$. Let
$S:=\{2, \sqrt{-3}\}$ and
\begin{eqnarray*}
T&=&\{ \sqrt{-3}-2,\ 
\sqrt{-3}+2,\  
1+2 \sqrt{-3},\ 
\sqrt{-3}-4,\
\sqrt{-3}+4,
  5+2\sqrt{-3},\ 5+4 \sqrt{-3}\}\\
  U&=&\{11, \sqrt{-3}-2, 1+2\sqrt{-3},\ 
     \sqrt{-3}-4 \}
\end{eqnarray*}
be three sets of primes in $\mathcal O_{K}$.
Suppose that $\rho_{1},\rho_{2}:\operatorname{Gal}(\bar
K/K)\lra\operatorname{GL}_{2}(E)$ are continuous Galois
representations unramified outside $S$ and
satisfying 
\begin{enumerate}
\item [1.] $\operatorname{Tr}(\rho_{1}(Frob_{\mathfrak p})) \equiv
  \operatorname{Tr}(\rho_{2}(Frob_{\mathfrak p})) \equiv 0
  \pmod{\mathcal P}$  for
  $\mathfrak p\in U$,
\item [2.]
  $\det(\rho_{1}) \equiv
  \det(\rho_{2})  \pmod{\mathcal P}$,
\item [3.] $\operatorname{Tr}(\rho_{1}(Frob_{\mathfrak p})) =
  \operatorname{Tr}(\rho_{2}(Frob_{\mathfrak p})) $
and
  $\det(\rho_{1}(Frob_{\mathfrak p})) =
  \det(\rho_{2}(Frob_{\mathfrak p}))  $ for
  $\mathfrak p\in T$.
\end{enumerate}
Then $\rho_{1}$ and $\rho_{2}$ have isomorphic semisimplifications.
\end{prop}

\begin{proof}
We claim that condition 1. implies $\operatorname{Tr}(\rho_{1}) \equiv
\operatorname{Tr}(\rho_{2}) \equiv 0 \pmod{\mathcal P}$. To prove this we have
to determine all cubic extensions of $\QQ[\sqrt{-3}]$, which are unramified
outside $S$; they give degree 6 extensions of $\QQ$
unramified outside $\{2,3\}$ as in the proof of Proposition \ref{prop1.1}. 
Out of the 398 
extensions 16 contain $\QQ[\sqrt{-3}]$;
they are splitting fields of the following degree six
polynomials
\begin{eqnarray*}
&&{x}^{6}-{x}^{3}+1,\quad
{x}^{6}-3\,{x}^{5}+5\,{x}^{3}-3\,x+1, \quad
{x}^{6}+3,\quad
{x}^{6}-3\,{x}^{5}+3\,{x}^{3}+6\,{x}^{2}-9\,x+3\\
&&{x}^{6}-3\,{x}^{4}-2\,{x}^{3}+9\,{x}^{2}+12\,x+4,\quad
{x}^{6}-2\,{x}^{3}+4,\quad
{x}^{6}-3\,{x}^{3}+3, \quad
   {x}^{6}+12,\quad
   {x}^{6}+48\\
&&{x}^{6}-3\,{x}^{4}+9\,{x}^{2}-18\,x+12, \quad
{x}^{6}-6\,{x}^{3}+12, \quad
{x}^{6}-6\,{x}^{3}+36,\quad
{x}^{6}+18\,{x}^{4}+81\,{x}^{2}+12\\
&&{x}^{6}+3\,{x}^{4}-2\,{x}^{3}+9\,{x}^{2}-12\,x+4,\quad
{x}^{6}+27\,{x}^{2}-36\,x+12,\quad
{x}^{6}-6\,{x}^{4}-4\,{x}^{3}+9\,{x}^{2}+12\,x+52  
\end{eqnarray*}
each of which factors into a product of degree 3 polynomials over
$\QQ[\sqrt{-3}]$.
One readily verifies that each degree 3 polynomial has 
irreducible reduction modulo at least one prime from $U$.
The evenness of the traces follows.

The compositum $K_{S}$ of all quadratic extensions of $\QQ[\sqrt{-3}]$ unramified
outside $\{2,\sqrt{-3}\}$ equals the compositum of $\QQ[\sqrt{-3},\sqrt2]$,
$\QQ[\sqrt[4]{-3}]$, $\sqrt{\tfrac12(\sqrt{-3}+1)}$.
From the table of quadratic characters below, we read off that the
elements from the Galois group $\operatorname{Gal}(K_{S}/\QQ[\sqrt{-3}])$
corresponding to Frobenius elements at the primes from $T$ form a non-cubic set.
Now the
proposition follows from \cite[Thm. 4.3]{L}.
\[
  \begin{array}{r|r|c|c|c||r|r|c|c|c}
    \mathfrak p&N(\mathfrak p)&\ \ \sqrt{-3}\ \ &\ 2\ &\tfrac12(\sqrt{-3}+1)& 
\mathfrak p&N(\mathfrak p)&\ \ \sqrt{-3}\ \ &\ 2\
    &\tfrac12(\sqrt{-3}+1)\\\hline
\sqrt{-3}-2&7&0&0&1&\sqrt{-3}+4&19&1&1&1\\
\sqrt{-3}+2&7&1&0&1&5+2\sqrt{-3}&37&0&1&0\\
1+2\sqrt{-3}&13&1&1&0&5+4\sqrt{-3}&73&1&0&0\\
\sqrt{-3}-4&19&0&1&1\\
  \end{array}
\]
\end{proof}

\section{Double octics}
\label{s:octics}

In this paper we shall study the modularity of three Calabi-Yau threefolds
constructed as crepant resolution of a double cover of the
projective space $\PP^{3}$ branched along an arrangement of eight planes
$S=P_{1}\cup\dots\cup P_{8}$. 

If the planes satisfy the following
two conditions:
\begin{eqnarray}
\label{eq:cond}
\text{no six planes intersect,
\;no four planes contain a common line,}
\end{eqnarray}
then the double cover admits a projective crepant resolution of
singularities.
One calls the resulting Calabi-Yau threefold  a \emph{double
  octic}. 
  It is sometimes useful to note that
  the crepant  resolution can be arranged in such a way that 
  it exhibits 
the double cover as a double cover of a blow-up of
the projective space.

One of the key features of double octics
is that one can control their invariants,
in particular their Hodge numbers.
In particular, this is instrumental for constructing rigid double octics
or one-dimensional families (accounting for all the infinitesimal deformations of the smooth members).
For brevity, we omit the details here
and refer the
reader to the section~4.2 of C.~Meyer's monograph \cite{Mey}.

\section{Double octic with real multiplication by $\QQ[\sqrt2]$}

Let $X$ be the double octic Calabi-Yau threefolds constructed as a
resolution of the double covering of $\PP^{3}$ 
branched along the following 8 hyperplanes:
\[u^{2}=x(x-z)(x-v)(x-z-v)y(y-z)(y-v)(y+v+2z).\]
By separating the variable $x, y$ on the right-hand side,
one realizes that the double octic $X$ admits a fibration by Kummer surfaces 
(the fibration
is induced by the map $(x,y,z,v,u)\mapsto(z,v)$).
Following \cite{Schoen}, this Kummer structure arises from
the fiber product of rational elliptic surfaces with
singular fibers $I_{4},I_{4},I_{2},I_{2}$ and
$I_{2},I_{2},I_{2},I_{2},I_{2},I_{2}$ where the singular fibers are located as follows:
\[\def\arraystretch{1.3}
  \begin{array}{cccccc}
    \infty&0&1&-1&-\frac 12&-\frac 13\\\hline
    I_{4}&I_{4}&I_{2}&I_{2}&&\\
    I_{2}&I_{2}&I_{2}&I_{2}&I_{2}&I_{2}
  \end{array}
\]
The Calabi-Yau threefold $X$ is isomorphic to the element corresponding to
$t=-1/2$ of the one parameter family defined by the Arrangement
No. 250 (\cite{Mey}). In particular
\begin{eqnarray}
\label{eq:hij}
h^{11}(X)=37, \;\;\; h^{12}(X)=1,
\end{eqnarray}
and
the only primes of bad reduction of $X$ are $2$ and $3$.
The  Riemann-symbol of the Picard-Fuchs operator of the family of
Calabi--Yau threefolds defined by the Arrangement No.~250 is 
\[
\left\{ 
\begin {array}{cccccc}
-2 &-1&-1/2&0&1&\infty\\
\hline
0&  0&0&  0&0&1/2\\
1&1/2&1&1/2&1&1/2\\
1&1/2&3&1/2&1&3/2\\
2&1  &4& 1 &2&3/2\\
\end {array} \right\} 
\]
(for details see \cite{CvS}). The Picard-Fuchs operator is symmetric
with respect to the involution $t\mapsto -1-t$ and its fixed point $t=-\frac
12$ is an apparent singularity. The family, 
however, does not seem to be symmetric in an obvious way.
Our findings will depend in an essential way
on
a correspondence between  members of
the family exchanged by the involution (\cite{CvS}).
Applied to the given Calabi--Yau threefold $X$,
the correspondence induces a two-to-one rational map 
$$
\Psi:X\lra X
$$  
defined over $\QQ[\sqrt2]$ by
\[\Psi:\
  \begin{pmatrix}
    x\\y\\z\\v\\u
  \end{pmatrix}
  \mapsto
  \begin{pmatrix}
 x\left( x-v-z \right)  \left( z-v \right)  \left( 3\,y+v \right) \\
\tfrac12\, \left( 3\,z+v \right)  \left( {v}^{2}-2\,xv+zv+2\,{x}^{2}-2\,xz \right)  \left( y-v \right) \\
\tfrac12\, \left( {v}^{2}-2\,xv+zv+2\,{x}^{2}-2\,xz \right) \left( 3\,y+v \right) \left( z+v \right)\\
\tfrac12\, \left( {v}^{2}-2\,xv+zv+2\,{x}^{2}-2\,xz \right) \left( 3\,y+v \right) \left( z-v\right)\\
\tfrac{\sqrt2}2\, \left( v-z \right)  \left( v+3\,y \right) ^{2}{v}^{2}
\left( 2\,x-v-z \right) 
\left( v+z \right)  \left( 3\,z+v \right)  \left( {v}^{2}-2\,xv+zv+2\,{x}^{2}
-2\,xz \right) ^{2}u
  \end{pmatrix}
\]
The pullback by $\Psi$ of a canonical form
$\omega_{X}$ is
$\Psi^{*}\omega_{X}=\sqrt2\omega_{X}$. In particular the map
$\Psi^{*}$ acts as multiplication by
$\sqrt{2}$ on $H^{3,0}(X)\oplus H^{0,3}(X)$. 
On the other hand, the map $\Psi^*$ acts as the multiplication by $(-1)$ on the infinitesimal
deformation space $H^{1}(\mathcal T_{X})$ and hence as multiplication by
$(-\sqrt{2})$ on $H^{1,2}(X)\oplus H^{2,1}(X)$ 
(by Serre duality there is an isomorphism
between vector spaces $H^{1,2}(X)$ and $(H^{1}(\mathcal
T_{X}) \otimes H^{3,0})^{*}$ compatible with the action induced by $\Psi$).
Consequently the map $\Psi$ decomposes the motive $H^{3}(X)$ into a
direct sum of two two-dimensional submotives 
\begin{eqnarray}
\label{eq:split}
H^3(X) = H^{3}_{+}\oplus H^{3}_{-}
\end{eqnarray}
 defined as  $(\pm\sqrt 2)$--eigenspaces of
$\Psi^{*}$. 
The restriction of the Galois 
action to the sub-group $\operatorname{Gal}(\bar\QQ/\QQ[\sqrt2])$
preserves the two submotives and hence decomposes $H^3(X)$
as the direct sum of
two Galois-conjugate Galois representations
\[\rho,\bar\rho:\operatorname{Gal}(\bar \QQ/\QQ[\sqrt2])\lra
  \operatorname{GL}_{2}(\QQ_{2}[\sqrt2]).\] 
We shall study these Galois representations using the Lefschetz fixed point formula
(in oder to eventually prove their modularity).
To this end, we have to inspect the crepant resolution of the double octic more closely.

The Calabi-Yau threefold $X$ is a double covering of a blow-up
$X'$ of the projective space $\PP^{3}$, consequently there is an
involution $i:X\lra X$ acting on $X$. This involution induces a
decomposition 
$$\mbox{Pic}(X)=H^{2}(X,\ZZ)=H^{2}_{sym}(X,\ZZ)\oplus
H^{2}_{skew}(X,\ZZ)$$ of the Picard group of $X$ into
symmetric and skew-symmetric part. The symmetric part
$H^{2}_{sym}(X,\ZZ)$ is isomorphic to the cohomology group
$H^{2}(X',\ZZ)$. 
The  octic arrangement defining the Calabi-Yau threefold $X$ has 28 double lines and 8 fourfold
points, consequently in the process of resolution of singularities of
$X$ we blow-up the doubly covered projective space $36$ times and the rank of the
cohomology group
$H^{2}(X',\ZZ)$ equals 37. 
It follows from \eqref{eq:hij}
that the cohomology group
$H^{2}(X,\ZZ)$ is generated by classes of symmetric divisors defined over
$\QQ$.
By the comparison theorem for any prime $p\ge5$ the Frobenius morphism
$\operatorname{Frob}_{p}$,   acts on the \'etale cohomology
$H_{\rm et}^{2}(\bar {X}_{p},\QQ_{l})$ by multiplication
with $p$, where ${X}_{p}$ is the reduction of $X$ modulo $p$ and $\bar X_p=X\otimes\bar\FF_p$,
and likewise for all powers $\Frob_q$.

In order to compute  the trace of the Frobenius morphism
$\operatorname{Frob}_{q}$ using the
Lefschetz fixed point formula, first we count the points $N_q$ on the
singular double octic over $\FF_{q}$ using a computer program, then we add the
correction terms for the crepant resolution of singularities. Over $\QQ$, the exceptional locus
of the blow-up of a fourfold point not contained in any triple line
(type $p_{4}^{0}$ in the notation of \cite{Mey}) is isomorphic to a surface 
\begin{eqnarray}
\label{eq:E}
E=\{u^{2}=\alpha xyz(x+y+z)\subset \PP^{3}(1,1,1,2)\}, \text{ where
  }\alpha\in \QQ.
  \end{eqnarray}
  The number of points on $E$ over $\FF_q$ depends on whether $\alpha$ is a square (for details see \cite[p.~56]{Mey});
  in particular, we have
$E(\FF_q) = q^{2}+2q+1$ for any even $p$-power $q$. 
Any
other blow-up adds $q^{2}+q$ points to $X$. By the Lefschetz fixed point
formula,  the trace of
$\operatorname{Frob}_{p^{2}}$ on $H_{\rm et}^{3}(\bar{X}_{p},\QQ_{l})$ 
thus equals
\[
a_{p^{2}}:=\operatorname{Tr}(\operatorname{Frob}_{p^{2}}|H^{3}(\bar{X}_{p}))=-N_{p^{2}} +
  p^{6}+p^{4}+9p^{2}+1
\]
By \eqref{eq:split},
the Galois representation on $H^{3}(X)$ equals its tensor product with the
Dirichlet character associated to the Legendre symbol $(\tfrac 2p)$.
Hence,
if $p$ is an \emph{inert} prime in $\QQ[\sqrt2]$,  the trace of
$\operatorname{Frob}_{p}$ on $H^{3}(X)$ vanishes:
\[
a_{p}:=\operatorname{Tr}(\operatorname{Frob}_{p}|H^{3}(X))=0;
\]
consequently the Frobenius polynomial equals
\[
X^{4}-\frac{a_{p^{2}}}2X^{2}+p^{6}.
\]
If $F_{p}\in Gal(\bar \QQ/\QQ[\sqrt2])$ is a Frobenius element, 
then the trace of value of $\rho$ and $\bar\rho$ at $F_{p}$ equals
$$\operatorname{Tr}(\rho(F_{p}))=\operatorname{Tr}(\bar\rho(F_{p}))=\frac12a_{p^{2}}.
$$ 

If $p$ is a \emph{split} prime,
 then the trace of $\operatorname{Frob}_{p^{2}}$
can be computed as before by a point count in $\FF_{p^{2}}$. In order
to compute the trace of $\operatorname{Frob}_{p}$ with a point-count
we have to take into account the contribution from the eight fourfold
points of the arrangement. Using \eqref{eq:E} we get
in this situation
\[a_{p}:=\operatorname{Tr}(\operatorname{Frob}_{p}|H^{3}(X))=-N_{p} +
  p^{3}+p^{3}+p(2+3(\tfrac{-1}p)+4(\tfrac{-2}p))+1.
\]
The Frobenius polynomial equals 
\[
\chi(\operatorname{Frob}_{p})=X^{4}-a_{p}X^{3} -
\tfrac12(a_{p}^{2}+a_{p^{2}})X^{2} -a_{p}p^{3} + p^{6}.
\]
In the following table we collect Frobenius polynomials for the values of $p$
that we will need to prove modularity.

\def\arraystretch{1.4}
\begin{longtable}[t]{|r|r|r|l|}
\hline $p$&$a_{p}$&$a_{p^{2}}$&$F_{p}$\\\hline
\hline 5 & 0 & 20 &\ppp {X}^{4}-10\,{X}^{2}+15625
 || \\
\hline 7 & 32 & -796 &\ppp {X}^{4}-32\,{X}^{3}+910\,{X}^{2}-10976\,X+117649\\
({X}^{2}+4\,\sqrt {2}X-16\,X+343)\times ({X}^{2}-4\,\sqrt {2}X-16\,X+343)|| \\
\hline 11 & 0 & -1452 &\ppp {X}^{4}+726\,{X}^{2}+1771561\\
({X}^{2}-44\,X+1331)\times ({X}^{2}+44\,X+1331)|| \\
\hline 17 & -124 & -10940 &\ppp {X}^{4}+124\,{X}^{3}+13158\,{X}^{2}+609212\,X+24137569\\
({X}^{2}+16\,\sqrt {2}X+62\,X+4913)\times ({X}^{2}-16\,\sqrt {2}X+62\,X+4913)|| \\
\hline 23 & 80 & -45212 &\ppp {X}^{4}-80\,{X}^{3}+25806\,{X}^{2}-973360\,X+148035889\\
({X}^{2}+8\,\sqrt {2}X-40\,X+12167)\times ({X}^{2}-8\,\sqrt {2}X-40\,X+12167)|| \\
\hline 31 & 272 & -59068 &\ppp {X}^{4}-272\,{X}^{3}+66526\,{X}^{2}-8103152\,X+887503681\\
({X}^{2}-76\,\sqrt {2}X-136\,X+29791)\times ({X}^{2}+76\,\sqrt {2}X-136\,X+29791)|| \\
\hline 41 & 84 & -148252 &\ppp {X}^{4}-84\,{X}^{3}+77654\,{X}^{2}-5789364\,X+4750104241\\
({X}^{2}-176\,\sqrt {2}X-42\,X+68921)\times ({X}^{2}+176\,\sqrt {2}X-42\,X+68921)|| \\
\hline 47 & -64 & -134460 &\ppp {X}^{4}+64\,{X}^{3}+69278\,{X}^{2}+6644672\,X+10779215329\\
({X}^{2}+264\,\sqrt {2}X+32\,X+103823)\times ({X}^{2}-264\,\sqrt {2}X+32\,X+103823)|| \\
\hline 89 & -2476 & 507556 &\ppp {X}^{4}+2476\,{X}^{3}+2811510\,{X}^{2}+1745503244\,X+496981290961\\
({X}^{2}+256\,\sqrt {2}X+1238\,X+704969)\times ({X}^{2}-256\,\sqrt {2}X+1238\,X+704969)|| \\
\hline 97 & 1284 & -2822268 &\ppp {X}^{4}-1284\,{X}^{3}+2235462\,{X}^{2}-1171872132\,X+832972004929\\
({X}^{2}+32\,\sqrt {2}X-642\,X+912673)\times ({X}^{2}-32\,\sqrt {2}X-642\,X+912673)|| \\
\hline
\end{longtable}
To avoid working with
four-dimensional Galois representations (as in \cite{DPS}),
we have determine the precise traces of $\operatorname{Frob}_{\mathfrak
  p}$ on $H^{3}_{+}$ and $H^{3}_{-}$ for a prime of $\OO_K$ above $p$.
To this end, we shall 
exploit the endomorphism $\Psi$;
more precisely, we
study the action of $\operatorname{Frob}_{\mathfrak
  p}\circ\Psi$ on $H^3(X)$.
This map preserves the Kummer fibration and transforms the fiber at
$(z,v)$ into the fiber at $(z+v,z-v)$. This allows us to determine
the number of fixed points of $\Psi$;
indeed, we can restrict ourselves to
the fibers at  $(1\pm\sqrt2,1)$. At those points, the fiber is
isomorphic to the
Kummer surface of the product of the elliptic curves
$$
u^{2}=x^{3}-30x+56\;\;\; \text{ and } \;\;\; u^{2}=y^{3}-y,
$$ 
and the map $\Psi$
is induced by the complex multiplications given
\begin{eqnarray*}
x\longmapsto -\frac{x^{2}-4x+18}{2(x-4)}  
\;\;\; \text{ and } \;\;\; 
y\longmapsto -\frac{y+1}{y-1}.
\end{eqnarray*}

As the map $\Psi$ acts on $H^{0,3}\oplus H^{3,0}$ as multiplication by
$\sqrt2$ and on $H^{2,1}\oplus H^{1,2}$ as multiplication by
$-\sqrt2$, we infer that the trace of the induced map on the third
cohomology is zero: $\tr(\Psi^{*}|H^{3})=0$.
Using Magma we computed that the map $\Psi$ has Lefschetz number
equal 12, so we get
\[\tr(\Psi^{*}|H^{0})=1, \;\; \tr(\Psi^{*}|H^{2})+\tr(\Psi^{*}|H^{4})=9, \;\; \tr(\Psi^{*}|H^{6})=2.\]
In a similar way we computed the trace of the composition
$\operatorname{Frob}_{\mathfrak p}\circ\Psi$ for split primes
$p=7,17,23, 31, 47, 89$. 
As the Picard group of $X$ is defined by divisors defined
over $\QQ$, the Frobenius morphism $\Frob_\p^*$ acts on $H^{2k}$ as  multiplication by
$p^{k}$, $k=0,1,2,3$. By direct computations with Magma,
 we found  the
Lefschetz numbers listed in the table below.
\[
\begin{array}{r||r|r|r|r|r|r}
  \mathfrak p&3+\sqrt2&3-\sqrt2
  &3\sqrt2-1
  &5+\sqrt2&5-\sqrt2&4\sqrt2+1\\\hline
  N(\mathfrak p)&7&7&17&23&23&31\\\hline
  \mathcal L(\operatorname{Frob}_{\mathfrak p}\circ\Psi)&944&976
  &11404&27104&27040&64208\\\hline
  \tr(\operatorname{Frob}_{\mathfrak p}^{*}|H^{3}_{+}) & 
  16+4\sqrt2&16-4\sqrt2
  &-62-16\sqrt2&40-8\sqrt2&40+8\sqrt2&
  136+76\sqrt2
  \\\hline\hline
\mathfrak p&4\sqrt2-1&5\sqrt2-3&\sqrt2+7 &\sqrt2-7
  &4\sqrt2-11&1-7\sqrt2\\\hline
  N(\mathfrak p)&31&41&47&47&89&97\\\hline
  \mathcal L(\operatorname{Frob}_{\mathfrak p}\circ\Psi)
             &64816&147116&219936&217824&1450924&1872652\\\hline
  \tr(\operatorname{Frob}_{\mathfrak p}^{*}|H^{3}_{+})
  &136-76\sqrt2&42-176\sqrt2&-32-264\sqrt2&-32+264\sqrt2&-1238-256\sqrt2&
  642+32\sqrt2
\end{array}
\]
The table also lists the traces of $\operatorname{Frob}_{\mathfrak
  p}\circ\Psi$ on $H^3(X)$.
  These can be determined as follows.
Since we do not know which factor of the Frobenius polynomial $F_{p}$
 corresponds to the characteristic polynomial of $\Frob_\p$ on
$H^{3}_{+}$ and which to $H^{3}_{-}$, we can determine the trace of
$\operatorname{Frob}_{\mathfrak p}\circ\Psi$ a priori only up to a sign. 
From the table on page 7 we get the following values of traces of
$\operatorname{Frob}^{*}_{\mathfrak p}$ on $H^{3}_{+}/H^{3}_{-}$
\[
  \begin{array}{c|c|c|c|c|c|c|c}
    7&17&23&31&41&47&89&97\\\hline
    16\pm4\sqrt2&-62\pm16\sqrt2&40\pm8\sqrt2&136\pm76\sqrt2
              &42\pm176\sqrt2& -32\pm264\sqrt2& -1238\pm256\sqrt2& 642\pm32\sqrt2
  \end{array}
\]

We have for any split prime $p\in\ZZ$ and any prime $\mathfrak
p\in\ZZ[\sqrt2]$ over $p$

\[\mathcal L(\operatorname{Frob}^{*}_{\mathfrak p}\circ\Psi)
  =
  1+p\operatorname{tr}(\Psi^{*}|H^{2})-\sqrt2(\operatorname{tr}(\operatorname{Frob}^{*}_{\mathfrak
    p}|H^{3}_{+}) - \operatorname{tr}(\operatorname{Frob}^{*}_{\mathfrak
    p}|H^{3}_{-}))+p^{2}\operatorname{tr}(\Psi^{*}|H^{4})+2p^{3}.
\]
and
\[\operatorname{tr}(\Psi^{*}|H^{2})+\operatorname{tr}(\Psi^{*}|H^{4})=9.\]
In the case $p=7,\;\;\mathfrak p=3+\sqrt2$ we get two possibilities
\[976=1+7\operatorname{tr}(\Psi^{*}|H^{2}) -
  \sqrt2((16-4\sqrt2)-(16+4\sqrt2)) + 49\operatorname{tr}(\Psi^{*}|H^{4}) +686
\]
or 
\[976=1+7\operatorname{tr}(\Psi^{*}|H^{2}) -
  \sqrt2((16+4\sqrt2)-(16-4\sqrt2)) + 49\operatorname{tr}(\Psi^{*}|H^{4}) +686,
\]
or equivalently,
\[273=7(\operatorname{tr}(\Psi^{*}|H^{2})+7\operatorname{tr}(\Psi^{*}|H^{4}))
  \qquad\text{ and }\qquad
305=7(\operatorname{tr}(\Psi^{*}|H^{2})+7\operatorname{tr}(\Psi^{*}|H^{4}))  
\]
As $7\nmid305$, the second option is impossible and consequently
\[\operatorname{tr}(\Psi^{*}|H^{2})+7\operatorname{tr}(\Psi^{*}|H^{4})=39.\]
Together with 
\[\operatorname{tr}(\Psi^{*}|H^{2})+\operatorname{tr}(\Psi^{*}|H^{4})=9,\]
this
yields
\[\operatorname{tr}(\Psi^{*}|H^{2})=4, \qquad \operatorname{tr}(\Psi^{*}|H^{4})=5.\]

Now, we get
\[\sqrt2(\operatorname{tr}(\operatorname{Frob}^{*}_{\mathfrak
    p}|H^{3}_{+}) - \operatorname{tr}(\operatorname{Frob}^{*}_{\mathfrak
    p}|H^{3}_{-})) = -\mathcal L(\operatorname{Frob}^{*}_{\mathfrak
    p}\circ\Psi)+1+4p+5p^{2}+2p^{3},
  \]
  and we can compute the entries of the above table.


Using Magma one finds 3 Hilbert modular forms for $K=\QQ[\sqrt2]$ of
weight $[4,2]$ and level $6\sqrt2\mathcal O_K$.
For one of them, let us denote it by $h_{1}$,
the Hecke eigenvalues agree exactly with
the traces of the action of Frobenius  on $H_{+}^{3}$ computed in
the above table.
(For the reader's convenience, we provide a table of eigenvalues at \cite{www}.)

\begin{thm}
\label{thm1}
  The Galois representation of $\operatorname{Gal}(\bar \QQ/\QQ[\sqrt2])$
  on the motive $H^{3}_{+}$ is isomorphic to the Galois representation
  of the Hilbert modular form $h_{1}$ for $K=\QQ[\sqrt2]$ of weight $[4,2]$ and
  level $6\sqrt2\mathcal O_K$. 
\end{thm}

\begin{proof}
  There exist continuous Galois representations
  \[
  \rho_{1},\rho_{2}: \;\;
  \operatorname{Gal}(\bar \QQ/\QQ[\sqrt{2}]) \lra
    \operatorname{GL}_{2}(\QQ_{2}[\sqrt{2}]) \]
  defined by the motive $H_{+}^{3}$ and the Hilbert modular form
  $h_{1}$. We shall verify that the representations $\rho_{1}$ and
  $\rho_{2}$ satisfy the assumptions of Proposition~\ref{prop1.1}.
  We have computed the traces of $\operatorname{Frob}_{\mathfrak
      p}|H^{3}_{+}$ for $\mathfrak p\in T$ and verified with MAGMA
    that they agree with Hecke eigenvalues of $h_{1}$;
    for $\mathfrak
    p \in U$ we check that both traces are even. Moreover for any
   $\p\in T$ we check that $\det(\rho_{1}(\operatorname{Frob}_{\mathfrak
     p}))=p^{3}$.
Since $\det(\rho_{1}(\operatorname{Frob}_{\mathfrak
      p}))|p^{6}$ for any $p\ge5$ and any prime $\mathfrak p$ in
    $\QQ[\sqrt2]$ over $p$, it follows that $\det(\rho_{1}(\operatorname{Frob}_{\mathfrak
      p}))$ is odd.

    Finally, $h_{1}$ is a Hilbert modular form with trivial character,
    so $\det(\rho_{2}(\operatorname{Frob}_{\mathfrak
      p}))=N(\mathfrak p)^{3}$ (which is odd).
      Thus the assumptions of Proposition~\ref{prop1.1} are satisfied,
            and  applying the proposition concludes the proof.
\end{proof}

\begin{remark}
\label{rem'}
It follows that 
 the Galois representation 
  on the motive $H^{3}_{-}$ is isomorphic to the Galois representation
  of the Hilbert modular form $\bar h_1$ for $\QQ[\sqrt2]$ of weight $[2,4]$ and
  level $6\sqrt2\mathcal O$. 
Observe the  divisibility condition $\bar a_\p\in\p$ for all Hecke eigenvalues of $\bar h_1$
in the given range of primes (in agreement with \cite[\S 3]{Hida}).
We believe that this could be proven geometrically using the Hodge type $(2,1)+(1,2)$ of $H^3_-$ 
along the lines of  \cite{Mazur}.
This could then also simplify the determination of the factor of the characteristic polynomial of Frobenius
corresponding to $H^3_+$ at ordinary primes.
\end{remark}

\section{Hilbert modular rigid Calabi-Yau threefold over
  $\QQ[\sqrt5]$}

Let $Y$ be the double octic defined as a crepant resolution of singularities
of the hypersurface
\[u^{2}=xyzv \left( x+y+z \right)  \left(  \varphi y-z+v
 \right)  \left( x+y+ \varphi v \right) \left( 
 (1-\varphi) x+y-\varphi z+ \varphi v \right)\subset \PP(1,1,1,1,4),\]
where  $\varphi=\tfrac12(-1+\sqrt{5}).$ Then $Y$ is a rigid Calabi-Yau
threefold with $h^{11}=38$ by \cite[Prop.~5.4]{CS-non-liftable}, and one verifies as before that the Picard group is generated
by divisors defined over $K=\QQ[\sqrt5]$ while the only prime of bad
reduction of $Y$ is 2. 

For prime numbers $p\equiv 1,4 \pmod5$ we  computed the  numbers
$n_{p}$ and $n_{p^{2}}$ of points of the singular double covering over
$\mathbb F_{p}$ and over $\mathbb F_{p^{2}}$. The resolution of
singularities is blowing-up 28 double lines and 9 fourfold points, and
by the Lefschetz fixed point formula we get the following traces on $H^3(\bar Y_\p)$
(where we suppress the cohomology group for ease of notation): 
\[\operatorname{Tr}(\operatorname{Frob}_{\mathfrak p^{2}}) =
  -n_{\p^{2}}+p^{6}+p^{4}-8p^{2}+1, \quad 
  \operatorname{Tr}(\operatorname{Frob}_{\mathfrak
    \p})=-n_{p}+p^{3}+p^{2}+cp+1, \;\;\;
    c\in\{-8,\dots,10\},\]
in a similar way as for the Calabi-Yau threefold $X$.
Moreover, comparing the actions of $\operatorname{Frob}_{\mathfrak p}$
and $\operatorname{Frob}_{\mathfrak p^{2}}$
\[
\operatorname{Tr}(\operatorname{Frob}_{\mathfrak p^{2}})
  = \operatorname{Tr}(\operatorname{Frob}_{\mathfrak p})^{2}-2p^{3}.
  \]
We  observe that  these equalities often suffice to determine 
$\operatorname{Tr}(\operatorname{Frob}_{\mathfrak p})$.
The next table  collects the results of the computations for the split primes
which we shall
need in the proof of modularity.
\[
\begin{array}{|r|r|r|r|r|r|r|}
\hline  p&\mathfrak p&\varphi&n_{\p}&n_{\p^{2}}&\operatorname{Tr}(\operatorname{Frob}_{\mathfrak p})&
\operatorname{Tr}(\operatorname{Frob}_{\mathfrak p^{2}})\\
\hline
  \hline 11&\sqrt5+4 &3& 1459& 1784297& 60& 938\\
  \cline{2-7} &\sqrt5-4&7& 1461& 1786601& 36& -1366\\
  \hline 29&2\sqrt5+7& 5& 25217& 595525129& -218& -1254\\
  \cline{2-7} &2\sqrt5-7&23& 25089& 595564553& -90& -40678\\
  \hline 31&\sqrt5+6& 12& 30685& 888442233& 192& -22718\\
  \cline{2-7}&\sqrt5-6 &18& 31003& 888475001& -64& -55486\\
  \hline 61&2\sqrt5-9& 17& 230471& 51534519081& 354& -328646\\
  \cline{2-7}&2\sqrt5+9 &43& 230215& 51534272297& 610& -81862\\
\hline
\end{array}
\]
To apply Proposition \ref{prop2}, we also require information at two inert primes.
For $p=3$ we were able to compute the number of points
of the singular double cover of
$\PP^{3}$
 over
$\FF_{3^{2}}$ and $\FF_{3^{4}}$,  obtaining $n_{9}=815$, $n_{81}=538617$. Similar computations
as in the case of split primes give
$\operatorname{Tr}(\operatorname{Frob}_{9})=-1262$ and 
$\operatorname{Tr}(\operatorname{Frob}_{3})=14$.
For $p=13$ we computed that the number of points over $\FF_{13^{2}}$
equals $4857691$.
To obtain the contribution 
 for the exceptional divisors \eqref{eq:E} over the fourfold points,
notice that the values of  $\alpha$ in question,
\[6-2\,\sqrt {5},\ 
\tfrac12(-3+3\,\sqrt {5}),\ 
\tfrac12(-3+3\,\sqrt {5}),\ 
4\,\sqrt {5}-8,\ 
-6\,\sqrt {5}+14,\ 
\tfrac12(3-1\,\sqrt {5}),\ 
6\,\sqrt {5}-14,\ 
\sqrt {5}-3,\ 
-1+\sqrt {5},\]
are all squares in $\FF_{13^2}$.
From the Lefschetz fixed point formula  we thus derive the trace of $\operatorname{Frob}_{13}$ 
on $H^3(\bar Y_{13})$ as 
\[
\operatorname{Tr}(\operatorname{Frob}_{13})=
-(4857961+9\cdot 13^{2})+1+13^{2}+13^{4}+13^{6}=-3942.\]

Using Magma we found 24 Hilbert modular forms for $\QQ[\sqrt5]$ of
parallel weight $[4,4]$ and level $16\mathcal O$;
one of them, denoted by $h_{2}$,
has exactly the same Hecke eigenvalues as the Frobenius traces above.  
Proposition \ref{prop2} thus guarantees the modularity of $Y$:

\begin{thm}
\label{thm2}
  The Galois representation of $\operatorname{Gal}(\bar \QQ/\QQ[\sqrt5])$ on
  $H^{3}_{et}(Y,\QQ_{l})$ is Hilbert modular with corresponding
  Hilbert modular form $h_2$.
  \end{thm}

\section{Modular Calabi-Yau threefold over
  $\QQ[\sqrt{-3}]$}

Let $Z$ be the double octic defined as a resolution of singularities
of the hypersurface 
\begin{eqnarray}
\label{eq:Z}
u^{2}=xyzv \left( x+y \right)  \left( x+y+z-v \right)  \left(
    \zeta\,x-y+\zeta\,z \right)  \left(y -\zeta\,z-v \right)\subset
  \PP(1,1,1,1,4),
  \end{eqnarray}
where $\zeta=\tfrac12(-1+i\sqrt{3})$ and, of course, $i^2=-1$. Then $Z$ is a rigid Calabi-Yau
threefold with $h^{11}=46$ as can be checked by considering it as a member of 
the one-dimensional family of double octics given by arrangement No.\ 262 in \cite{Mey}.
As before, one verifies that the Picard group is generated
by divisors defined over $K=\QQ[\sqrt{-3}]$, and that the only prime of bad
reduction of $Z$ is 2. 
\begin{prop}
  $Z$ is birational to a Calabi-Yau threefold defined over $\QQ[i]$. 
\end{prop}

\begin{proof}
The standard crepant resolution of a double octic proceeds as follows:
 blow-up successively  fivefold
points, triple lines, fourfold points and double lines. 
The
resolution depends on the order of blow-ups of double lines; to
overcome this subtlety we modify the last step and blow-up the union
of all double lines in the singular double cover (cf. \cite{CS}). 
Then the map 
\[(x,y,z,v,u)\mapsto(\zeta\,x,-\zeta\,x-\zeta\,y,- \left( \zeta+1
  \right) x-y- \left( \zeta+1 \right) z,- \left( \zeta+1 \right)
  \left( x+y+z-t \right) ,iu)\] 
defines an isomorphism of $Z$ and its Galois conjugate over $K$, hence
$Z$ isomorphic to a variety defined over $\mathbb Q[i]$ by the Weil
Galois Descent Theorem (\cite[Thm.~3]{Weil}). 
\end{proof}

We can count points over $\mathbb F_{p}$ only if $p\equiv 1\pmod6$,
i.e. $p$ is a split prime in $K$. Above a given split prime $p$ there are two prime ideals
$\mathfrak p$ in the ring of integers of $\QQ[\sqrt{-3}]$;
this corresponds to two choices for $\zeta\in\FF_\p$ and two
possibilities for the trace of Frobenius on $H^3(\bar Z_\p)$ which we list in the next table.

\[
\begin{array}{|r|r|r|r|r|r|}
\hline p&\zeta&
\operatorname{Tr}(\operatorname{Frob}_{\mathfrak p})&
\zeta&\operatorname{Tr}(\operatorname{Frob}_{\mathfrak p})\\ 
\hline 7 &     4 & -12  &   2  &12    \\
\hline 13&     3 & -58   &   9  &-58  \\
\hline 19&     11& -136 &   7  &136   \\
\hline 31&     25& 20   &   5  &-20   \\
\hline 37&     26& -18  &   10 &-18   \\
\hline 43&     6 & -200 &   36 &200   \\
\hline 61&     47& -458  &   13 &-458 \\
\hline 67&     29& -496 &   37 &496   \\
\hline 73&     64& 602 &   8  &602    \\
\hline 79&     55& 1108 &   23 &-1108 \\ 
\hline 97&     61& -206  &   35 &-206 \\
\hline
\end{array}
\]
The computed traces agree up to sign with the Fourier coefficients of a
modular form $f$ of weight 4 for $\Gamma_{0}(72)$
:(72/1 in Meyer's notation in \cite{Mey}):
\[\begin{array}{r||r|r|r|r|r|r|r|r|r|r|r}
    p&7&13&19&31&37&43&61&67&73&79&97\\\hline
    a_{p}& -12& 58&-136& 20& -18&  -200& -458& -496& -602& 1108&  206
\end{array}
\]
We compare the signs with characters from the table in the proof of
Proposition~\ref{prop1.3} to notice that the sign changes are governed
by the character corresponding to the extension $\QQ[\sqrt[4]{-3}]/\QQ[\sqrt{-3}]$. 

\begin{thm}
\label{thm3}
  Consider the Galois representation of $\operatorname{Gal}(\bar \QQ/\QQ[\sqrt{-3}])$ on
  $H^{3}_{et}(\bar Z,\QQ_{l})$  
  and the one
  associated to the modular form $f$ 
  restricted to $\QQ[\sqrt{-3}]$ and then twisted by the quadratic character associated to
  the extension $\QQ[\sqrt[4]{-3}]/\QQ[\sqrt{-3}]$. 
  Then the Galois representations have isomorphic semi-simplifications.
\end{thm}

\begin{proof}
In view of Proposition~\ref{prop1.3},
compared to the present data,
it suffices to check the following two properties:
\begin{itemize}
\item
 $\sqrt{-3} = 2\zeta+1$ is a square in $\FF_p$
if and only if the corresponding choice of $\operatorname{Tr}(\operatorname{Frob}_{\mathfrak p})$
matches the Fourier coefficient $a_p$ of $f$;
\item
the Galois representations have even trace at $p=11$.
\end{itemize}
The latter condition follows easily since $a_{11}=64$ and
any double octic Calabi--Yau threefold given by an arrangement of 8 planes
satisfying  condition \eqref{eq:cond}
 is checked to have an even number of points over any finite field of odd parity
 by an elementary combinatorial argument.
\end{proof}


\end{document}